\documentclass[12pt]{amsart}

\usepackage{pstricks-add}
\usepackage{pst-all}

\usepackage{amsmath,amssymb,latexsym}
\usepackage{amsfonts,amstext,amsthm,amscd}
\usepackage{graphicx, color}
\usepackage{enumerate}
\theoremstyle{plain}

\usepackage{color}

\newtheorem{theorem}{Theorem}[section]
\newtheorem{proposition}[theorem]{Proposition}
\newtheorem{definition}[theorem]{Definition}
\newtheorem{lemma}[theorem]{Lemma}
\newtheorem{corollary}[theorem]{Corollary}
\newtheorem{remark}[theorem]{Remark}
\newtheorem{example}[theorem]{Example}

\newcommand{\A}{\mathbb A}
\newcommand{\Af}{\mathbb{A}_f}
\newcommand{\aid}{\mathfrak a}
\newcommand{\Ann}{\mathrm{Ann}}
\newcommand{\dA}{d_{\mathbb{A}_f}}
\newcommand{\C}{\mathbb C}
\newcommand{\car}{\mathfrak 1}
\newcommand{\Char}{\mathrm{Char}}
\newcommand{\Co}{\mathrm C}

\newcommand{\D}{\mathcal D}
\newcommand{\Dkl}{\D^\ell_k(\Af)}

\newcommand{\F}{\mathcal F}

\newcommand{\lcm}{\mathrm{lcm}}

\newcommand{\N}{\mathbb N}
\newcommand{\Pp}{\mathbb P}
\newcommand{\Q}{\mathbb Q}
\newcommand{\Qc}{\overline \Q}
\newcommand{\Qp}{\Q_p}
\newcommand{\R}{\mathbb R}

\newcommand{\Real}{\mathrm{Re}}
\newcommand{\UC}{\mathbb{S}^1}
\newcommand{\Z}{\mathbb Z}
\newcommand{\Zp}{\Z_p}
\newcommand{\Zprod}{\prod_{p\in \Pp} \Zp}
\newcommand{\Zz}{\widehat{\Z}}

\newcommand{\abs}[1]{\left\vert#1\right\vert}
\newcommand{\mTo}{\longmapsto}
\newcommand{\norm}[1]{\left\Vert#1\right\Vert}
\newcommand{\normAf}[1]{\left\Vert#1\right\Vert_{\Af} }
\newcommand{\ord}{\mathrm{ord}}
\newcommand{\To}{\longrightarrow}

\date{\today}

\begin{document}

\title[Harmonic Analysis on the Ad\`ele Ring of $\Q$]{Harmonic Analysis on the Ad\`ele Ring of $\Q$}
\author[V.A. Aguilar--Arteaga, M. Cruz--L\'opez and S. Estala--Arias]
{Victor A. Aguilar--Arteaga$^*$, Manuel Cruz--L\'opez $^{**}$ and Samuel Estala--Arias $^{***}$}

\address{$^{*}$ Departamento de Matem\'aticas, CINVESTAV, Unidad Quer\'etaro, 
M\'exico} 
\email{aguilarav@math.cinvestav.mx}

\address{$^{**}$ Departamento de Matem\'aticas, Universidad de Guanajuato,
Jalisco S/N Mineral de Valenciana, Guanajuato, Gto.\ C.P.\ 36240, M\'exico} 
\email{manuelcl@ugto.mx}

\address{$^{***}$ IMATE-JU, UNAM, M\'exico} 
\email{samuelea82@gmail.com}



\begin{abstract} 

The ring of finite ad\`eles $\Af$ of the rational numbers $\Q$ is obtained in this article as a completion of $\Q$ with respect to a certain non--Archimedean metric. This ultrametric allows to represent any finite ad\`ele as a series generalizing $m$--adic analysis. The description made provides a new perspective for the adelic Fourier analysis on $\Af$, which also permits to introduce new oscillatory integrals. By incorporating the infinite prime, the mentioned results are extended to the complete ring of ad\`eles $\A$ of $\Q$.

\end{abstract} 
\maketitle 

\section{Introduction}
\label{introduction}

Since the introduction of the $p$--adic fields $\Qp$ by K. Hensel, the development of 
$p$--adic theory has been a very fruitful subject of study. For each prime number $p$, 
$\Qp$ is a non--Archimedean completion of the rational numbers $\Q$. Putting together all these completions with the Archimedean one, $\R$, the ring of ad\`eles $\A$ of $\Q$ is constituted. The ring of ad\`eles contains complete information on the arithmetics of the rational numbers.

The fields $\Qp$ are non--Archimedean local fields which are complete with respect to an ultrametric whose unit ball centred at zero  is the maximal compact and open subring 
$\Zp$, the ring of $p$--adic integers. The ad\`ele ring $\A$ is  the direct product 
$\R\times \Af$, where $\Af$ is the so called ring of finite ad\`eles and it constitutes the totally disconnected part of $\A$. The ring of finite ad\`eles $\Af$ is classically defined as the restricted direct product of the fields $\Qp$, over all prime numbers, with respect to the corresponding rings of integers. The additive structure of $\Af$ and $\A$ makes them  locally compact Abelian topological groups whose harmonic analysis has played a fundamental role in algebraic number theory  (see \cite{Tat}, \cite{McF}, \cite{GM} \cite{Rob}). More recently, several studies of pseudodifferential equations, stochastic processes and others analytical tools on the ring of ad\`eles $\A$ and other adic subrings of $\A$ have been done (see \cite{CL}, \cite{CE}, \cite{DRK}, \cite{KKS}, \cite{KR}, \cite{Evd}, \cite{Urb}, \cite{Yas}, \cite{Zun}). 

Both ad\`eles rings $\A$ and $\Af$, being traditional objects in number theory, have   several equivalent descriptions which provide different models of them. For instance, $\Af$ can be seen as the smallest second countable locally compact topological ring which contains all $p$--adic completions of $\Q$. From this perspective, the basic goal of this article is twofold: to present $\Af$ as a general adic ring allowing series representations and to develop the adelic Fourier analysis on $\Af$ generalizing $m$--adic analysis (\cite{VVZ}, \cite{AKS}, \cite{Koc}, \cite{DZ}). In order to  extend these results to the complete ad\`ele ring $\A$, the infinite prime is considered. 
 
Adic series representations lead to describe $\Af$ as a second countable locally compact totally disconnected topological ring and to see it is an ultrametric space (\cite{HR}). Several other ultrametrics on $\Af$ and $\A$ have been considered in different works such as \cite{CE}, \cite{Str}, \cite{McF} and \cite{Zun}. The ultrametric chosen here follows from a general framework on locally compact totally disconnected Abelian groups (\cite{ACE}, \cite{SC}). This ultrametric allows to evaluate some oscillatory integrals as generalized Dirichlet series related to the second Chebyshev function (see Section \ref{characters_integration}).  
 
Series representations go back to the original construction by  Hensel and their introduction is due to Pr\"ufer, von Neumann, Novoselov, among others (see e.g., \cite{HR} and \cite{DZ}). The arithmetics of polyadic integer numbers have been studied carefully in several works (\cite{Bro}). In particular, the adelic topology of integers was considered by Furstenberg (\cite{Fur}). However, the case of the finite ad\`ele ring has not been considered before in complete detail. The case of the finite ad\`ele ring is a limit case in the sense that any other adic ring is contained on this ring. 
 
Let us summarize briefly the content of this article. To begin with the analysis, we introduce extended second Chebyshev and von Mangoldt functions $\psi(n)$ and 
$\Lambda(n)$ over $\Z$. These functions provide a totally directed system of neighborhoods of zero $\{e^{\psi(n)}\Z\}_{n\in\Z}$ for an additive invariant topology on 
$\Q$ whose completion is known to be the finite ad\`ele ring (Compare with \cite{GGPH}). With these ingredients, any non-zero finite ad\`ele $x$  can be written as an adic convergent series 
\[ x = \sum_{k=\gamma}^\infty a_k  e^{\psi(k)}, \qquad (\gamma \in \Z)  \]
with $a_\gamma\neq 0$ and $a_k\in \{0,1,\ldots,e^{\Lambda(k+1)} - 1\}$, which allow us to state the general properties of $\Af$ as in the pioneering work by V. S. Vladimirov  \cite{Vla}. A fractional part function $\{\cdot\}_{\Af}:\Af \To \Q$ is introduced in order to define a canonical additive character $\chi(x)=e^{2\pi i \{x\}_{\Af}}$ on $\Af$. With this character at hand it is shown that $\Af$ is autodual in the sense of Pontryagin. Additionally, the chosen filtration provides a natural description of the Bruhat--Schwartz spaces of test functions $\D(\Af)$ as nuclear spaces and permits the study of the Fourier transform, the Fourier inversion formula and the Parseval--Steklov identity on 
$\D(\Af)$ and $L^2(\Af)$. 

Considering the place at infinity $\R$, a description of the ring of ad\`eles $\A$ as a locally compact Abelian topological ring is given. The Bruhat--Schwartz space $\D(\A)$ is an inductive limit of Fr\'echet spaces which is isomorphic to the algebraic and tensor product $\D(\R)\otimes \D(\Af)$, where $\D(\R)$ is the Schwartz space of $\R$. The 
Fourier transform on $\A$, given by $\F_{\A}=\F_{\R}\otimes \F_{\Af}$ is an isomorphism on $\D(\A)$. The Fourier inversion formula and the Parseval--Steklov identity are satisfied.  

The space of square integrable functions $L^2(\A)$ on $\A$ is a separable Hilbert space since it is the Hilbert tensor product space 
\[ L^2(\A) \cong L^2(\R) \otimes L^2(\Af) \] 
and $L^2(\R)$ and $L^2(\Af)$ are separable Hilbert spaces (see \cite{RS}). From this decomposition, the following fundamental and classical result is obtained.

\noindent \textbf{Theorem: \label{fourier_transformA}}
The Fourier transform $\F:L^2(\A)\To L^2(\A)$ is an isometry and the Fourier inversion formula and the Parseval--Steklov identities hold on $L^2(\A)$.

Finally, it is very important to say that given any completion of the rational numbers by a filtration similar to the one considered here conduces to an analogous harmonic analysis on $\Af$, and therefore on $\A$.
 
The manuscript is arranged as follows. In Section \ref{finite_adeles} the ring $\Af$ is introduced as a completion of the rational numbers $\Q$. The additive group of characters of $\Af$ and the basic theory of functions and integration is done in Section \ref{characters_integration}. The Bruhat--Schwartz spaces and the theory of Fourier analysis are presented in Section \ref{fourier_analysisAf}. Finally, in Section \ref{fourier_analysisA} the theory is done on the complete ring of ad\`eles.

\section[The finite ad\`ele ring $\Af$ as a completion of $Q$]{The finite ad\`ele ring $\Af$ as a completion of $\Q$}
\label{finite_adeles}

In this Section the finite ad\`ele ring $\Af$ is depicted as a completion of the rational numbers $\Q$. First, an arithmetical function related to the second Chebyshev function is introduced in order to describe a family of additive subgroups and an ultrametric on $\Q$. The completion of the rational numbers with respect to this ultrametric is shown to be
a second countable locally compact topological Abelian ring isomorphic to the finite 
ad\`ele ring. For a comprehensive introduction to the theory of  the ring of ad\`eles we quote the books \cite{Wei}, \cite{RV}, \cite{Lan} and the classical treatises \cite{GGPH} and \cite{Tat}. 

\subsection[An extended second Chebyshev function]{An extended second Chebyshev function}
\label{extended_Chebyshev}

Denote by $\N=\{1,2,\ldots\}$ the set of natural numbers and let $\Pp\subset \N$ be the set of prime numbers. Recall the following classical arithmetical functions (see \cite{Apo}). Let $\psi(n)$ denote the \textsf{second Chebyshev function} defined by the relation 
\[ e^{\psi(n)} = \lcm(1,2,\ldots,n) \qquad (n\in \N), \]
where $\lcm(a,b)$ denotes the least common multiple of $a$ and $b$.
Write $\Lambda(n)$ for the \textsf{von Mangoldt function} given by 
$$
\Lambda(n) = 
\begin{cases} 
\log p & \text{ if } n=p^k \text{ for some } p\in \Pp \text{ and integer } k\geq 1,\\ 
0 & \text{otherwise}.
\end{cases}
$$ 
These arithmetical functions are related by the equivalent equations
$$ 
\psi(n) = \sum_{k=1}^n \Lambda(k) \quad \text{and} \quad 
e^{\psi(n)} = \prod_{k =1}^n e^{\Lambda(k)}. 
$$ 

For any integer number $n$, define the second symmetric Chebyshev function by
\begin{equation*}
\psi(n) = 
\begin{cases} 
\frac{n}{\abs{n}} \psi(\abs{n}) &  \text{ if } n\neq 0, \\
0  & \text{ if } n=0, 
 \end{cases}
\end{equation*}
and the symmetric von Mangoldt function by (extending) the relation
\[ e^{\Lambda(n)} = \frac{e^{\psi(n)}}{e^{\psi(n-1)}}. \]
That is,
\begin{equation*}
\Lambda(n) = 
\begin{cases} 
\Lambda(n) & \text{ if } n > 0, \\
\Lambda(\abs{n-1})=\Lambda(\abs{n}+1) & \text{ if } n\leq 0. 
\end{cases}
\end{equation*}

Notice that if $e^{\psi(n)} < e^{\psi(n+1)}$ then $e^{\psi(n+1)}/e^{\psi(n)}$ is a prime number given by $e^{\Lambda(n+1)}$; otherwise, $e^{\psi(n)} = e^{\psi(n+1)}$ and $e^{\Lambda(n+1)}=1$. The values where the first case occurs form the set
$$ 
\mathcal{C} = \left\{ -p^{l} : l\in \N,\; p\in \Pp \right\} \cup 
\left\{ p^{l}-1 : l \in \N,\; p\in\Pp \right\}. 
$$  
For the purposes of the article the following change of parameter on these values is useful:

\begin{remark}
\label{ramification_index}
Let $\rho:\Z \To \mathcal{C}$ be the unique increasing bijective function such that $e^{\psi(\rho(n))}$ is a strictly increasing function with $e^{\psi(\rho(0))}=1$.
 
In order to simplify notation we will write $e^{\psi(n)}$ and $e^{\Lambda(n)}$ instead of  $e^{\psi(\rho(n))}$ and $e^{\Lambda(\rho(n))}$, respectively. With this notation, 
$e^{\psi(n+1)}/e^{\psi(n)}$ is a positive prime number $e^{\Lambda(n+1)}$. Moreover, for any integers $n > m$, the functions $\psi(n)$ and $\Lambda(n)$ satisfy the relations
\begin{equation*}
\label{ramificacion}
\psi(n) - \psi(m) = \sum_{k=m+1}^{ n} \Lambda(k) \quad \text{ and } \quad 
e^{\psi(n)}/e^{\psi(m)} = \prod_{k =m+1}^n e^{\Lambda(k)}. 
\end{equation*}
\end{remark}

\begin{remark}
\label{cofinalN}
From the definition of the extended second Chebyshev function, it follows that the sequences $(e^{\psi(n)})_{n=0}^\infty$ and $(1/e^{\psi(-n)})_{n=0}^\infty$ coincides, are totally ordered by division and cofinal with $\N$, i.e. given any natural number $N$, there exists $l\in \N$ such that $N$ divides $e^{\psi(l)}$.
\end{remark}

The following observation is of great importance to the construction given on the next section.

\begin{remark}
Using a similar procedure as the one used to write any rational number with respect to a given integer base, any positive rational number $q$ admits a unique representation as a finite sum
\[ q = \sum_{k=\gamma}^N a_k e^{\psi(k)}, \qquad (q\in \Q_{+}, \,\gamma \in \Z) \]
with $a_\gamma\neq 0$ and $a_k\in \{0,1,\ldots,e^{\Lambda(k+1)} - 1\}$.
\end{remark}

\subsection[The ring of finite adelic numbers]{The ring of finite adelic numbers}
\label{adelic_numbers}

To each integer number $n$, there corresponds an additive subgroup 
$e^{\psi(n)}\Z$ of the rational numbers $\Q$. If $n$ is nonnegative, $e^{\psi(n)}\Z$ is an ideal of $\Z$, otherwise if $n$ is negative, $e^{\psi(n)}\Z$ is a fractional ideal of $\Q$. The family of all such subgroups is totally ordered by inclusion. In fact, the filtration
$$
\{0\} \subset \cdots \subset e^{\psi(n)}\Z \subset \cdots \subset \Z 
\subset \cdots \subset e^{\psi(m)}\Z \subset \cdots \subset \Q \qquad (m<0<n)
$$  
has the properties 

\begin{equation*}
\label{union-intersection}
\bigcap_{n\in \Z} e^{\psi(n)}\Z = \{0\}\quad \text{ and } \quad 
\bigcup_{n\in \Z} e^{\psi(n)}\Z = \Q.
\end{equation*}

The collection $\{ e^{\psi(n)}\Z \}_{n\in \Z}$ is a countable neighborhood base of zero for a second countable additive invariant topology on $\Q$. In fact, since $\Q$ is countable, the topology on $\Q$ is generated by a numerable collection of open sets. This topology is called here the \textsf{finite adelic topology} of $\Q$. Since the intersection of the neighborhood base at zero consists only of the point zero, the finite adelic topology on 
$\Q$ is Hausdorff. 

From Remark \ref{cofinalN}, it follows that  the induced topology on $\Z$ coincides with the Furstenberg topology and any nonempty open set is given by the union of arithmetic progressions (see \cite{Fur}).  Furthermore, Remark \ref{cofinalN} implies that the finite adelic topology of $\Q$ is the topology generated by all additive subgroups of $\Q$ (see \cite{GGPH}). Hence, a nonempty subset $U\subset \Q$ is open if it is a union of rational arithmetics progressions in $\Q$, i.e. $U$ is a union of sets of the form $l\Z+q$, where 
$l$ and $q$ are fixed rational numbers. This implies that a set $U$ is open if and only if  given any $q\in \Q$ there exists a nonzero rational number $l$ such that 
$l\Z+q \subset U$. Notice that arithmetic progressions are both open and closed because
\[ l\Z+q =\Q \setminus \bigcup _{\substack{  j \in \Q  \\ 0<j<l }} l\Z+q+j. \]

A sequence $(a_n)_{n\in \N}$ of rational numbers is a \textsf{Cauchy sequence} in the finite adelic topology if for all $k\in \Z$ there exists $N>0$ such that, if $n,m >N$, then 
$a_n-a_m \in e^{\psi(k)}\Z$, i.e. if $e^{\psi(k)}$ divides $a_n-a_m$ in the sense that there exists an integer number $c$ such that $ce^{\psi(k)}=a_n-a_m$. For instance, constant sequences with all their elements equal to a fixed rational number, are Cauchy sequences. In particular, when the fixed constant is zero, it is called the zero sequence. A less trivial example of a Cauchy sequence is a sequence given by the partial sums 
\[ S_n=\sum_{k=\gamma}^{n} a_k e^{\psi(k)} \] 
of any formal  infinite series of the form 
\[ x = \sum_{k=\gamma}^\infty a_k e^{\psi(k)}, \qquad (\gamma \in \Z) \]
with $a_k\in \left\{ 0,1,\ldots,e^{\Lambda(k+1)} - 1 \right\}$ and $a_\gamma \neq 0$. In fact, if $n\geq m$,

\begin{align}
S_n-S_{m}	&= \sum_{k=m+1}^{n} a_k  e^{\psi(k)} \in e^{\psi(m)}\Z.    \notag
\end{align}

Let $\mathcal{C}(\Q)$ be the set of all Cauchy sequences. With the sum and multiplication by components, $\mathcal{C}(\Q)$ is a commutative ring with unitary element the constant sequence $(1)_{n\in \N}$. Two Cauchy sequences $(a_n)_{n \in \N}$ and $(b_n)_{n \in \N}$  of rational numbers are said to be equivalent if $(a_n-b_n)_{n\in \N}$ converges to the zero sequence, i.e. for all $l\in \Z$
\[ a_n-b_n \in e^{\psi(l)}\Z, \] 
for $n$ sufficiently large. A Cauchy sequence is trivial if it converges to zero. The set of all trivial Cauchy sequences $\mathcal{C}_0(\Q)$ is an ideal of the commutative ring 
$\mathcal{C}(\Q)$. The completion $\Qc$ of $\Q$ is the quotient ring 
$\mathcal{C}(\Q)/\mathcal{C}_0(\Q)$ with the topology given by the concept of Cauchy convergence. It follows that $\Qc$ is a complete topological space where the natural inclusion of $\Q$, with the finite adelic topology in $\Qc$, is dense.

\subsubsection[Adic series]{Adic series} 

Let $(a_n)_{n \in \N}$ be a Cauchy sequence of rational numbers. If $(a_n)_{n \in \N}$ is not equivalent to zero in $\Qc$, there exists a maximum element $\gamma \in \Z$ such that there is a subsequence $(b_n)_{n\in \N}$ with $b_n \notin e^{\psi(\gamma)}\Z$. Since $(b_n)_{n \in \N}$ is a Cauchy sequence, there is another subsequence $(c_n)_{n \in \N}$ such that $c_n \notin e^{\psi(\gamma)}\Z$ and $(c_n-c_{n-1}) \in e^{\psi(\gamma)}\Z$. Therefore, there exists a unique integer 
$l_{\gamma-1} \in \{1,\ldots,e^{\Lambda(\gamma)} - 1 \}$ such that 
$c_n-l_{\gamma-1}e^{\psi(\gamma-1)} \in e^{\psi(\gamma)}\Z$ for all $n$. Once more, since $c_n-l_{\gamma}e^{\psi(\gamma-1)}$ is a Cauchy sequence, there exists a subsequence $d_n$ of $c_n$ such that 
\[ (d_n-d_{n-1} ) \in e^{\psi(\gamma+1)}\Z. \] 
Hence, there exists a unique integer $l_{\gamma}$ such that 
$l_{\gamma} \in \{0,1,\ldots,e^{\Lambda(\gamma+1)} - 1 \}$ and 
$$
d_n - l_{\gamma-1}e^{\psi(\gamma-1)} - l_{\gamma} e^{\psi(\gamma)} \in e^{\psi(\gamma+1)}\Z
$$ 
for all $n$. Inductively, the partial sums 
\[ \left(\sum_{k=\gamma-1}^{\gamma-1+n} l_k  e^{\psi(k)}\right)_{n\in \N} \] 
form a sequence representing $(a_n)_{n\in \N}$ in $\Qc$. Therefore,
every nonzero element $x\in \Qc$ has a representative, also denoted by $x$, which can be uniquely written as a convergent series 
\[ x = \sum_{k=\gamma}^\infty l_k  e^{\psi(k)}, \qquad (\gamma \in \Z) \]
with $l_k\in \left\{ 0,1,\ldots,e^{\Lambda(k+1)} - 1 \right\}$, $l_\gamma \neq 0$.

\begin{remark}
Due to series representations, the ring $\Qc$ is called here the \emph{\textsf{ring of finite adelic numbers}}.  
\end{remark}

\subsubsection[Arithmetical operations]{Arithmetical operations}

From the definitions of addition and product of Cauchy sequences the following holds. Let $x,y\in \Qc$ be two finite ad\`eles with series representation
$$
x=\sum_{k=\gamma}^{\infty} a_k e^{\psi(k)} \quad \text{and}\quad 
y=\sum_{k=\gamma}^{\infty} b_k e^{\psi(k)}  
$$ 
and let $S_k(x)$ and $S_k(y)$ be the partial sums 
$$
S_k(x)=\sum_{l=\gamma}^{k} a_l e^{\psi(l)} \quad \text{and} \quad 
S_k(y)=\sum_{l=\gamma}^{k} b_l e^{\psi(l)}.
$$
The coefficients $\{c_k\}_{k=\lambda}^{\infty}$ of the series representation of the sum  
\[ z=x+y=\sum_{k=\gamma}^{\infty} c_k e^{\psi(k)} \] 
are given as follows: since 
$e^{\psi(\gamma)} e^{\Lambda(\gamma+1)} = e^{\psi(\gamma+1)}$, there exists a unique $c_\gamma \in \{0,1,\ldots, e^{\Lambda(\gamma+1)}-1 \}$ such that
$$
c_\gamma e^{\psi(\gamma)} \equiv a_{\gamma}e^{\psi(\gamma)} + b_{\gamma}e^{\psi(\gamma)} \mod e^{\psi(\gamma+1)}.
$$
Inductively, there exists $c_k \in \{0,1,\ldots, e^{\Lambda(k+1)} -1\}$, ($k\geq \gamma$), such that
\[ \sum_{l=\gamma}^{k} c_l e^{\psi(l)} = S_k(x)+S_k(y)\mod e^{\psi(k+1)}. \]
 
In addition, the coefficients $\{c_k\}_{k=\lambda}^{\infty}$ of the series representation of the product  
\[ z=xy=\sum_{k=\gamma}^{\infty} c_k  e^{\psi(k)} \] 
can be found inductively by considering
\[\sum_{l=\gamma}^{k} c_l  e^{\psi(l)}  = S_k(x)S_k(y)\mod  e^{\psi(k+1)}, \]
where $c_k \in \{0,1,\ldots , e^{\psi(k+1)}-1 \}$.

The sum of any two finite ad\`eles can be done by the procedure of `carrying'  if necessary in the series representation, similar to classical decimal expansions. The multiplication however uses a similar procedure but with a bit more subtle algorithm (see \cite{CE}).

\begin{proposition}
\label{addition_series}
Let $x,y\in \Qc$ be two finite ad\`eles with series representation
$$
x=\sum_{k=\gamma}^{\infty} a_k  e^{\psi(k)} \quad \text{and}\quad 
y=\sum_{k=\gamma}^{\infty} b_k  e^{\psi(k)}.  
$$ 
The coefficients $\{c_k\}_{k=\lambda}^{\infty}$ of the series representation of the sum  
\[ z=x+y=\sum_{k=\gamma}^{\infty} c_k  e^{\psi(k)} \] 
are given by
\[ c_k = a_k+b_k+ r_{k-1}\mod  e^{\Lambda(k+1)}, \]
where $c_k \in \{0,1,\ldots, e^{\Lambda(k+1)}-1 \}$,
$r_{\gamma-1}=0$ and, (for $k=\gamma, \ldots$), $r_k$ is one or zero depending on whether $a_k+b_k +r_{k-1}$ is greater than $e^{\Lambda(k+1)}$ or not. 
\end{proposition}

\begin{proof} 
Consider the truncated series representation of $x$ and $y$:
$$
S_N(x)=\sum_{k=\gamma}^{N} a_k  e^{\psi(k)} \quad \text{and} \quad 
S_N(y)=\sum_{k=\gamma}^{N} b_k  e^{\psi(k)}.
$$

Hence
\begin{align*}
S_N(x)+S_N(y) &= 
\sum_{k=\gamma}^{N} a_k  e^{\psi(k)} + \sum_{k=\gamma}^{N} b_k  e^{\psi(k)}  \\
&= \sum_{k=\gamma}^{N} (a_k+b_k) e^{\psi(k)}\\
&= \sum_{k=\gamma}^{N} c_k e^{\psi(k)} + r_{N} e^{\psi(N)}.
\end{align*} 

Therefore $S_N(x)+S_N(y)$ converges to $z=x+y$ when $N\to \infty$ and   
\[ z = \sum_{k=\gamma}^{\infty} c_ke^{\psi(k)}. \]
\end{proof}
     
\begin{remark} 
The difference of any two finite ad\`eles is computed by the procedure of `taken' if necessary.
\end{remark}
 
\begin{example}
\[ -1 = \sum_{n=0}^{\infty} (e^{\Lambda(n+1)} - 1)  e^{\psi(n)}. \]
\end{example}

\subsubsection[An adelic ultrametric]{An adelic ultrametric}
\label{adelic_ultrametric}

From Proposition \ref{addition_series} 
the function $\ord:\Qc \To \Z\cup \{\infty\}$ given by  
\[ \gamma(x)=\ord(x) := \min \{ k \, : \, a_k \neq 0 \} \] 
defines an \textsf{order} on $\Qc$, in the sense that it satisfies  properties:
\begin{enumerate} 
\label{order}
\item $\ord(x) \in \Z\cup \{\infty\}$, and $\ord(x)=\infty$ if and only if $x=0$,
\item $\ord(x+y)\geq \min \{\ord(x),\ord(y)\}$, for any $x,y\in \Qc$.
\end{enumerate} 

This allows to introduce a non--Archimedean metric on $\Qc$ given by
\[ d_{\Qc}(x,y) = e^{-\psi(\ord(x-y))} \qquad (x,y\in \Qc). \]

In fact, the second property of the function $\ord(\cdot)$ above implies
\[ d_{\Qc}(x,z)\leq \max\{d_{\Qc}(x,y),d_{\Qc}(y,z)\}, \]
for any $x,y,z \in \Qc$. The ultrametric $d_{\Qc}$ takes values in the set 
$\{ e^{\psi(n)} \}_{n\in \Z} \cup \{0\}$ and the balls centred at zero are the completions of the original filtration
$$
B(0,e^{\psi(n)}) = \left\{ x\in\Qc : x=\sum_{k=\gamma}^\infty a_k e^{\psi(k)} \text{ with } n\leq \gamma \right\}, \qquad (n \in \Z). 
$$
 
\begin{definition}  
The collection of \emph{\textsf{finite adelic integers}} $\Zz$ is the unit ball of $\Qc$:
\[ \Zz = \left\{ x \in \Qc : x = \sum_{k=0}^{\infty} a_k e^{\psi(k)} \right\}. \]
\end{definition}

From the definition of addition and multiplication, it follows that the finite adelic integers 
$\Zz$ is a compact and open subring of $\Qc$. It is also maximal, because 
$\normAf{x^2}>\normAf{x}$ for any $x\in \Qc$ which is not in $\Zz$. In addition, $\Zz$ contains the set of integer numbers as a dense subset. 

\begin{remark} 
\label{profinite}
From the construction of $\Qc$ and the fact that $e^{\psi(n)}\Z$ is an ideal of $\Z$, it follows that $\Zz$ is the \emph{\textsf{profinite ring completion}} of the integers $\Z$. Moreover, for any other case, $e^{\psi(n)}\Zz$ is the profinite group completion of the additive subgroup $e^{\psi(n)}\Z$ (see \cite{Wil}).   
\end{remark}

From a geometrical point of view, the non--Archimedean property shows that any point of a ball is a center and that two balls are disjoint or one is contained in the other. This property can be described algebraically as follows. For each $n\in \Z$, define 
$\aid^n = e^{-\psi(n)} \Zz$.\footnote{It is important to notice that these are not powers of the ring $\Zz$.} The collection $\{\aid^n\}_{n\in \Z}$ is a neighborhood base of zero for the finite adelic topology on $\Qc$. Each subgroup $\aid^n$ provides a disjoint partition
\[ \Qc = \bigcup_{x\in \Qc/ \aid^n} (x + \aid^n) \] 
where the union is taken over a complete (countable) set of representatives of the quotient 
$\Qc/\aid^n$.

For each $n\in \Z$ and $x\in \Qc$, let $B_n(x)$ be the ball with center at $x$ and radius $e^{\psi(n)}$. From the non--Archimedean property, each element of a ball is its center and any ball is compact and open. The balls $B_n(x)$ coincides with the sets $x+\aid^n$ and the decomposition above can be written as
\[ \Qc = \bigcup_{x\in \Qc / \aid^n } B_n(x), \] 
where the union is taken over a complete set of representatives of the quotient 
$\Qc / \aid^n$. 

A complete set of representatives of the quotient $\aid^n /\aid^{n-1}$ determines the partition 
\[ \aid^n = \bigcup_{x\in \aid^n /\aid^{n-1} } (x+\aid^{n-1}) \]
which induces a corresponding partition
\[ \Qc = \bigcup_{ \substack{y \in \Qc / \aid^n,\\ x \in \aid^n /\aid^{n-1}} } (y+x+\aid^{n-1}). \]

\begin{remark}
An equivalent way to describe $\Qc$ is the following: the function
\[ \gamma(q)= \min\left\{\,l \, : \, q \in e^{\psi(l)\Z}\,\right\} \] 
is an order over $\Q$ in the sense that it satisfies properties \ref{order} and the function 
\[ d_{\Q}(p,q) = e^{-\psi(\ord(p-q))}, \] 
for $p,q \in \Q$, is an ultrametric on $\Q$. By construction, the ultrametric $d_{\Q}$ is the restriction of $d_{\Qc}$ to $\Q$ and therefore the metric completion of $\Q$ with respect to $d_{\Q}$ is naturally isometric to $\Qc$. 
\end{remark}

As a nice consequence of writing any element of $\Qc$ as a convergent series, it is shown the following result characterising the compact subsets of $\Qc$.

\begin{proposition} 
The topological ring $\Qc$ satisfies the Heine--Borel property, i.e. a subset 
$K\subset \Qc$ is compact if and only if it is closed and bounded. Therefore $\Qc$ is a locally compact topological ring. 
\end{proposition}

\begin{proof}  
Since $\Qc$ is an ultrametric space, it is sufficient to prove that any closed and bounded set $K\subset \Qc$ is sequentially compact. Let 
$(x_k)_{k\geq 1}$ be a bounded sequence in $\Qc$ and write
$$
x_k = \sum_{l=\gamma(x_k)}^{\infty} a_k(l) e^{\psi(l)} \qquad 
\big( a_k(l) \in  \left\{ 0,1,\ldots,e^{\Lambda(l+1)} - 1 \right\} \big).
$$ 

Since $(x_k)_{k\geq 1}$ is bounded, the set $\{\gamma(x_k)\}$ is bounded from below, say by $\gamma_{0}$. If the set $\{\gamma(x_k)\}$ is unbounded from above, then 
$(x_k)_{k\geq 1}$ has a subsequence that converges to zero. If the set $\{\gamma(x_k)\}$ is bounded from above, it only takes a finite number of values. Therefore, for some 
$\gamma_0 \in \Z$, there exists a subsequence of $(x_k)_{k\geq 1}$ with elements of the form

$$
a_0 e^{\psi(\gamma_0)} + a_1 e^{\psi(\gamma_0+1)} + a_2 e^{\psi(\gamma_0+2)} + \cdots \qquad (a_l \in \{0,1,2,\ldots, e^{\Lambda(\gamma_0+l)}-1 \}, \, a_{0} \neq 0 ).
$$
Moreover, since $a_0$ can take only a finite number of values, there exist a definite 
\[ a_0\in \{1,2,\ldots, e^{\Lambda(\gamma_0)}-1 \} \] 
and a subsequence of the form 
$$
a_0 e^{\psi(\gamma_0)} + a_1 e^{\psi(\gamma_0+1)} + a_2 e^{\psi(\gamma_0+2)} + \cdots \qquad (a_l \in \{0,1,2,\ldots, e^{\Lambda(\gamma_0+l)}-1 \}, \, a_0 \neq 0).
$$
Since any other $a_k$ can take only a finite number of values, inductively, it is found a subsequence which converges to a nonzero limit point 
$$
x = \sum_{l=\gamma_0}^{\infty} a_l e^{\psi(l)} \qquad 
\big( a_l) \in  \left\{ 0,1,\ldots,e^{\Lambda(l+1)} - 1 \right\}\, a_0 \neq 0 \big).
$$
\end{proof}

\begin{remark} 
For any nonzero rational numbers $p$ and $q$, $p\Z \cdot q\Z = pq\Z$. Therefore, for any nonzero rational number $l$ the preimage of $l\Z$ under the multiplication 
$\Q\times\Q \To \Q$ is the union of all $p\Z \times q\Z \subset \Q^2$ such that $pq=l$. Since this union is an open set, any completion of the rational numbers by a filtration as above produces a topological ring.  
\end{remark}

\begin{remark} 
Since $\Qc$ is a topological ring, for any nonzero rational number $q$, the set $q\Zz$ is a compact and open subgroup of $\Qc$. It can be seen that any compact and open subgroup of $\Qc$ is of the form $q\Zz$ for some nonzero rational number $q$.
\end{remark}

\subsubsection[A Haar measure on finite adelic numbers]{A Haar measure on finite adelic numbers}
\label{haar_measure}

Consider the additive function $dx$ on balls given by 
$$ 
dx(B_n(x)) = dx(\aid^n) := e^{\psi(n)} =
\begin{cases}
[\aid^n : \Zz] & \text{ if } n > 0, \\
[\Zz : \aid^n] & \text{ if } n \leq 0.
\end{cases}
$$
for any $x\in \Qc$ and $n\in \Z$. The set of all balls of positive radius and the empty set  forms a semiring and the additive function $dx$ is a $\sigma$--finite premeasure. By Carath\'eodory's extension theorem, there exists a unique Borel measure on $\Qc$, also denoted by $dx$, which extends this formula. Since balls are compact and open, it follows that $dx$ is a Radon measure on $\Qc$. Moreover, by construction $dx$ is additive invariant and hence it is a Haar measure on $\Qc$.

By definition, this measure assigns total mass one to $\Zz$. Moreover, the following statement holds.
\begin{proposition}
\label{ChangeVariables} 
For any rational number $q$, the Haar measure of $q\Zz$ is equal to $q^{-1}$.  
\end{proposition}

\begin{proof}
It is only necessary to compute the index for the subgroups $l\Zz$ for any natural integer $l$. First, $l\Z \cap e^{\psi(n)}\Z=\text{mcd}(l,e^{\psi(n)})\Z=l\Z$ for $n$ sufficiently large. The completion of $\l\Z$ is $l\Zz$ and
\[ [\Zz:l\Zz] = [\Z:l\Z]=l. \]
\end{proof}

\subsubsection[Non--Archimedean topology of the ring of finite ad\`eles numbers]{Non--Archimedean topology of the ring of finite ad\`eles numbers}
\label{adelic_topology}

In summary, $\Qc$ is a second countable, locally compact, totally disconnected, commutative, topological ring. The following properties holds (see picture \ref{figuraA}):

\begin{enumerate}
\item The inclusion $\Q\hookrightarrow \Qc$ is dense and $\Qc$ is a separable topological space.
\item $\Qc$ has a non--Archimedean metric $d_{\Qc}$ such that $(\Qc,d_{\Qc})$ is a complete ultrametric space and therefore it is a totally disconnected topological space.
\item A Haar measure on $\Qc$ can be chosen in such a way that the maximal compact and open subring $\Zz$ has total mass one.
\item The ultrametric $d_{\Qc}$ is the unique additive invariant ultrametric on $\Qc$ whose balls centered at zero is precisely the collection $\{\aid^n\}_{n\in \Z}$, and such that the Haar measure of any ball is equal to its radius. In particular, $\Zz$ has diameter one.
\item For any integer $n$, if $x\in B_n(y)$, $B_n(x)=B_n(y)$.
\item The set of  all balls of any positive fixed radius form a numerable partition of $\Qc$. 
\item The set of all balls of positive radius are numerable. Therefore, the topology of 
$\Qc$ is generated by a numerable base and the Borel $\sigma$--algebra of $\Qc$ is separable.  
\end{enumerate}

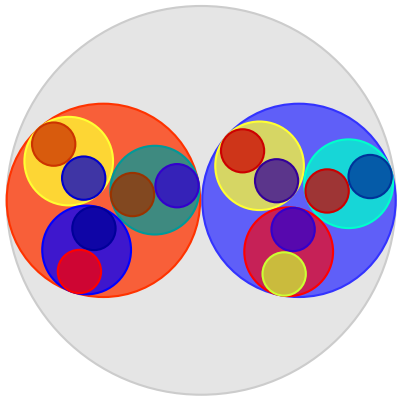
\begin{figure}[ht!]
\label{figuraA}
\begin{center}
\newrgbcolor{cccccc}{0.8 0.8 0.8}
\newrgbcolor{ttttff}{0.2 0.2 1}
\newrgbcolor{ffttqq}{1 0.2 0}
\newrgbcolor{fffftt}{1 1 0.2}
\newrgbcolor{qqzzzz}{0 0.6 0.6}
\newrgbcolor{qqffcc}{0 1 0.8}
\newrgbcolor{ccttqq}{0.8 0.2 0}
\newrgbcolor{qqqqcc}{0 0 0.8}
\newrgbcolor{zzttqq}{0.6 0.2 0}
\newrgbcolor{ttqqcc}{0.2 0 0.8}
\newrgbcolor{qqqqzz}{0 0 0.6}
\newrgbcolor{ccqqqq}{0.8 0 0}
\newrgbcolor{ttqqzz}{0.2 0 0.6}
\newrgbcolor{qqttzz}{0 0.2 0.6}
\newrgbcolor{ccfftt}{0.8 1 0.2}
\psset{unit=.10cm,algebraic=true,dotstyle=o,dotsize=3pt 0,linewidth=0.8pt,arrowsize=3pt 2,arrowinset=0.25}
\begin{pspicture*}(-34.59,-33.95)(38.2,32.11)
\pscircle[linecolor=cccccc,fillcolor=cccccc,fillstyle=solid,opacity=0.5](0,0){26}
\pscircle[linecolor=ttttff,fillcolor=ttttff,fillstyle=solid,opacity=0.75](13,0){13}
\pscircle[linecolor=ffttqq,fillcolor=ffttqq,fillstyle=solid,opacity=0.75](-13,0){13}
\pscircle[linecolor=fffftt,fillcolor=fffftt,fillstyle=solid,opacity=0.75](-17.61,5.22){6.03}
\pscircle[linecolor=qqzzzz,fillcolor=qqzzzz,fillstyle=solid,opacity=0.75](-6.17,1.36){6.04}
\pscircle[linecolor=blue,fillcolor=blue,fillstyle=solid,opacity=0.75](-15.23,-6.6){6.04}
\pscircle[linecolor=fffftt,fillcolor=fffftt,fillstyle=solid,opacity=0.75](7.76,4.6){6.03}
\pscircle[linecolor=qqffcc,fillcolor=qqffcc,fillstyle=solid,opacity=0.75](19.6,2.21){6.04}
\pscircle[linecolor=red,fillcolor=red,fillstyle=solid,opacity=0.65](11.63,-6.83){6.04}
\pscircle[linecolor=ccttqq,fillcolor=ccttqq,fillstyle=solid,opacity=0.75](-19.61,7.48){3.02}
\pscircle[linecolor=qqqqcc,fillcolor=qqqqcc,fillstyle=solid,opacity=0.75](-15.62,2.96){3.02}
\pscircle[linecolor=zzttqq,fillcolor=zzttqq,fillstyle=solid,opacity=0.75](-9.13,0.77){3.02}
\pscircle[linecolor=ttqqcc,fillcolor=ttqqcc,fillstyle=solid,opacity=0.75](-3.21,1.95){3.02}
\pscircle[linecolor=qqqqzz,fillcolor=qqqqzz,fillstyle=solid,opacity=0.75](-14.26,-3.74){3.02}
\pscircle[linecolor=red,fillcolor=red,fillstyle=solid,opacity=0.75](-16.19,-9.46){3.02}
\pscircle[linecolor=ccqqqq,fillcolor=ccqqqq,fillstyle=solid,opacity=0.75](5.5,6.59){3.01}
\pscircle[linecolor=ttqqzz,fillcolor=ttqqzz,fillstyle=solid,opacity=0.75](10.02,2.61){3.01}
\pscircle[linecolor=ccqqqq,fillcolor=ccqqqq,fillstyle=solid,opacity=0.75](16.74,1.25){3.02}
\pscircle[linecolor=qqttzz,fillcolor=qqttzz,fillstyle=solid,opacity=0.75](22.46,3.17){3.02}
\pscircle[linecolor=ttqqcc,fillcolor=ttqqcc,fillstyle=solid,opacity=0.75](12.22,-3.87){3.02}
\pscircle[linecolor=ccfftt,fillcolor=ccfftt,fillstyle=solid,opacity=0.7](11.03,-9.79){3.02}
\end{pspicture*}
\end{center}
\caption{The ring of finite ad\'eles}
\end{figure}

\subsection[Finite adelic numbers are isomorphic to the finite ad\`eles ring]{Finite adelic numbers are isomorphic to the finite ad\`eles ring }
\label{adelic_numbers-adeles}

In this paragraph the topological ring $\Qc$ is shown to be naturally isomorphic to the finite ad\`ele ring of $\Q$. For a complete description of this isomorphism see \cite{CE}.

\subsubsection[The ring of finite ad\`eles]{The ring of finite ad\`eles} 
\label{ring_Af} 

For each prime $p\in \Pp$ denote by $\Zp$ and $\Qp$ the ring of $p$--adic integers and the field of $p$--adic numbers, respectively. The \textsf{finite ad\`ele ring} $\Af$ of the rational numbers $\Q$ is classically defined as the restricted direct product of the fields 
$\Qp$ with respect to the maximal compact and open subrings $\Zp$. That is 
$$
\Af = \left\{ (x_p)_{p\in \Pp}\in  \prod_{p\in \Pp}  \Qp : x_p \in \Zp\; 
\text{for all but finitely many primes}\; p\in \Pp \right \}.
$$
with a topology given by the restricted direct product topology. 

The restricted direct product topology is the unique additive invariant topology generated by the neighborhood base of zero  
\[ \mathcal{N} = \left\{ q\Zprod : q\in \Q, \, q \neq 0 \right\}, \] 
which consists of compact and open subgroups of $\Af$, by Tychonoff's theorem. With the restricted direct product topology, the topological ring $\Af$ is second countable, totally disconnected, locally compact and commutative. 
 
The diagonal inclusion of $\Q$ into $\Af$ is dense, so $\Q$ has a unique topology whose completion $\Qc$ is isomorphic as a topological ring to $\Af$ as stated in the next result:

\begin{proposition}
\label{isomorphism_adeles} 
There exists an isomorphism of topological rings
\[ \Qc \cong \Af,\] 
which preserves the inclusion of $\Q$ in both rings.
\end{proposition}

\begin{proof} 
From Remark \ref{profinite} the ring of finite adelic integers $\Zz$ is the profinite completion of $\Z$.  By the Chinese Remainder Theorem, there exists an isomorphism of  topological rings between $\Zz$ and $\Zprod$, which preserves the natural inclusion of the integers on both rings. Furthermore, by the group version of the Chinese Remainder Theorem, for each any integer number $n$, there exists a topological group isomorphism between $e^{\psi(n)}\Zz$ and $e^{\psi(n)}\Zprod$, which preserves the natural inclusion of the group $e^{\psi(n)}\Z=e^{\psi(n)}\Zprod \cap \Q$. These isomorphisms are compatible, due to their natural inclusion of the groups $e^{\psi(n)}\Z$, and extends to an isomorphism from $\Qc$ to $\Af$, which preserves the natural inclusion of $\Q$ in both rings.
\end{proof}

\begin{remark} 
From Proposition \ref{isomorphism_adeles} it follows that the series representation of a finite ad\`ele $x$ is valid in $\Af$. That is to say, the partial sums
\[ s(x,N) = \sum_{l=\gamma}^N x_l e^{\psi(l)} \subset \Af \]
are convergent in $\Af$ with the restricted direct product topology. Using the ultrametric 
$\dA$, every finite ad\`ele $x\in \Af$ can be written as a series
\[ x = \sum_{l=\gamma}^\infty x_le^{\psi(l)}, \qquad (x_l \neq 0, l\in \Z) \]
with $x_l \in \{0,1,\ldots, e^{\Lambda(l+1)}-1\}$. This series is convergent in the ultrametric of $\Af$ and the numbers $x_l$ appearing in the representation of $x$ are unique and $\ord(x)=\gamma \in  \Z$. 
\end{remark}

\begin{remark} 
In the sequel we identify $\Qc$ with $\Af$ and set $\dA(x,y)=d_{\Qc}(x,y)$, for 
$x,y\in \Af$. We also write $\normAf{\cdot} := \dA(0,\cdot)$. 
\end{remark}

\section[Characters and integration theory on $\Af$]{Characters and integration theory on $\Af$}
\label{characters_integration}

This Section introduces the group of additive characters of the locally compact Abelian group $(\Af,+)$. The fractional part of an adic series allows to explicitly describe a canonical character of $\Af$ and to show that the group of finite ad\`eles is a selfdual group in the sense of Pontryagin. In addition, basic elements of integration theory on $\Af$ and some oscillatory integrals are discussed. The general theory of harmonic analysis on  locally compact Abelian topological groups has been collected in the monumental work \cite{HR}. The article \cite{Vla} and the book \cite{VVZ} describe this theory for the case of the field of $p$--adic numbers $\Qp$. We also quote references \cite{AKS} and \cite{Koc}.

\subsection[The group of additive characters of $\Af$]{The group of additive characters of $\Af$}
\label{charactersAf}

An \textsf{additive character} of the ring $\Af$ is a continuous group homomorphism 
$\chi:\Af\To \UC$, where $\UC$ is the multiplicative group of complex numbers with norm equal to one. Denote by $\Char(\Af)$ the topological group which consists of all continuous homomorphisms from $\Af$ into $\UC$ endowed with the compact and open topology. This group is called the \textsf{Pontryagin dual} of $\Af$ or, the \textsf{character group} of $\Af$. Since $\Af$ is locally compact, $\Char(\Af)$ is also a locally compact Abelian group.

\subsubsection[The canonical character]{The canonical character}
\label{canonical_characterAf}

According to Section \ref{adelic_numbers}, any finite ad\`ele $x\in \Af$ admits a unique series representation
\[ x = \sum_{k=\gamma(x)}^{\infty} a_k  e^{\psi(k)}, \] 
where $a_{\gamma(x)}\neq 0$ and $a_k \in \{0,1,\ldots,e^{\Lambda(k+1)} - 1\}$. If 
$\gamma(x) < 0$, $x$ can be decomposed as
\[ x = \sum_{k=\gamma(x)}^{-1} a_k  e^{\psi(k)} + \sum_{k=0}^{\infty} a_k  e^{\psi(k)}.\]
  
This leads to the following:

\begin{definition}
\label{fractional_part} 
The \emph{\textsf{fractional part}} of a finite ad\`ele $x\in \Af$ is
\begin{equation*}
\{x\} := 
\begin{cases} 
\displaystyle{\sum_{k=\gamma(x)}^{-1} a_k  e^{\psi(k)}}	& \text{ if } \gamma(x) < 0, \\
0 																							& \text{ if } \gamma(x) \geq 0. 
\end{cases}
\end{equation*}
\end{definition}

From this definition it is clear that $x\in \Zz$ if and only if $\{x\}=0$. Additionally, for any $x\in \Af$ with $\gamma(x) < 0$, its fractional part satisfies the inequalities:
\[ e^{\psi(\gamma)} \leq \{x\} \leq 1 -  e^{\psi(\gamma)} \qquad \big(\gamma=\ord(x)\big). \] 

From Proposition \ref{addition_series}, the fractional part function satisfies the relation
\[ \{x+y\} = \{x\} + \{y\} - N,  \qquad(x,y\in \Af) \]
where $N=0,1$. Therefore the map $x\mTo \exp (2\pi i\{x\})$ is a well defined group homomorphism from $\Af$ to the unit circle. Moreover, it is trivial on $\Zz$, which implies that it is a continuous function and a character of $\Af$. This observation allows to state the following definition.

\begin{definition}
\label{canonical_character}
The map $\chi: \Af \To \UC$ given by
\[ \chi(x) = \exp (2\pi i\{x\})\qquad (x\in \Af) \] 
is the \emph{\textsf{canonical additive character}} of $\Af$. 
\end{definition}

If $x\in \Zz$, $\chi(x)=1$; however, if $x\in \Af\backslash\Zz$, 
$\gamma=\gamma(x)<0$ and 
\begin{align*} 
\chi(x)	&= \exp (2\pi i\{x\}) \\
&= \exp \left( 2\pi i \sum_{k=\gamma}^{-1} a_k  e^{\psi(k)} \right) \\
&= \exp \left( \frac{2\pi i}{e^{-\psi(\gamma)}} \big( a_{\gamma}  + 
a_{\gamma+1}e^{\psi(\gamma+1) - \psi(\gamma)} + \cdots + 
a_{-1}e^{\psi(-1) -\psi(\gamma)} \big) \right) \\
&= \exp \left( \frac{2\pi i\ell}{e^{-\psi(\gamma)}} \right),
\end{align*}
where $\ell=a_{\gamma} + a_{\gamma+1}e^{\psi(\gamma+1)-\psi(\gamma)} + \cdots + a_{-1}e^{\psi(-1)-\psi(\gamma)}$ is an integer number smaller than $e^{-\psi(\gamma)}$. That is, the image of any element $x\in \Af$ under the canonical character is a nontrivial $N$--th root of unity for some $N\in \N$, whenever $x\notin \Zz$. This implies that the kernel of the canonical character is precisely $\Zz$. Since $\Af/\Zz$ is a discrete group, the image of $\Af$ under the canonical character is isomorphic to the group of roots of unity $\Q/\Z$ with the discrete topology. This can be subsumed as: 

\begin{remark} 
The kernel of the canonical character $\chi$ is the compact and open subgroup $\Zz$ of 
$\Af$ and $\chi$ is constant on the cosets of $\Zz$. Moreover, $\chi$ can be factorized as a composition
\begin{align*}
\Af \To \Af/\Zz \To \UC.
\end{align*}
 Therefore,
\[ \Af/\Zz \cong \Q/\Z, \]
where $\Q/\Z$ is the group of roots of unity with the discrete topology.
\end{remark}

\begin{remark}
The canonical additive character does not depend on the construction given here and it agrees with the classical canonical additive character of $\Af$. In fact, the fractional part of any $x\in \Af$ is the unique rational number $\{x\}\in \Q\cap [0,1)$ such that 
$x - \{x\} \in \Zz$ and there is a decomposition $\Af=\Q\cap [0,1)+\Zz$.
\end{remark}

\subsubsection[$\Af$ is selfdual]{$\Af$ is selfdual}
\label{selduality_Af}

Since multiplication is continuous on the ring $\Af$, for any $\xi \in \Af$, the function 
$\chi_{\xi}: \Af \To \UC$ defined as
\[ \chi_{\xi}(x) = \chi(\xi x) = \exp (2\pi i\{\xi x\}) \] 
is an additive character of $\Af$. For any sequence of ad\`eles $(\xi_n)_{n\in \N}$ converging to $\xi \in \Af$, the sequence of characters $(\chi_{\xi_n})_{n\in \N}$ converges uniformly on every compact subset of $\Af$. It follows that $\Char(\Af)$ contains $\Af$ as a closed subgroup. 

Observe now that if $\chi\in \Char(\Af)$ is any nontrivial additive character, then there exists an integer $n$ such that $\chi$ is trivial on $\aid^n$, but it is nontrivial on 
$\aid^{n+1}$. This can be seen in the following way: suppose that $V\subset \UC$ is an open neighborhood of the identity $1\in \UC$ such that no nontrivial subgroup of $\UC$ is contained in $V$. By continuity of $\chi$, the preimage $\chi^{-1}(V)$ is an open neighborhood of 0 in $\Af$, and therefore there exists $n\in \Z$ such that 
$\aid^n\subset \chi^{-1}(V)$. Being $\chi$ a homomorphism, this implies that 
$\chi(\aid^n)\subset V$ is a subgroup of $\UC$, and hence, it must be trivial. Since $\chi$ is nontrivial, there exists a minimum integer $n\in \Z$ such that $\chi$ is trivial on 
$\aid^n$, but it is nontrivial on $\aid^{n+1}$. This number $n$ is called the \textsf{rank} of the character $\chi$.

\begin{theorem}
\label{selfduality} 
The finite ad\`ele group $\Af$ is selfdual. That is to say, there exists an isomorphism of locally compact Abelian topological groups from $\Af$ to $\Char(\Af)$ given by 
$\xi \mTo \chi_\xi$.
\end{theorem}

\begin{proof}  
We have already seen that $\chi_\xi$ is an additive character for any $\xi \in \Af$ and that $\Af$ is closed in $\Char(\Af)$. Suppose that $\widetilde{\chi}$ is an arbitrary additive nontrivial character of $\Af$. If $n$ is the rank of the character $\widetilde{\chi}$, then 
$\widetilde{\chi}$ induces a nontrivial character on the quotient group 
$\aid^{n+1}/\aid^n$. Since $\chi_{e^{\psi(n)}}=\chi( e^{\psi(n)} \cdot )$ is trivial on 
$\aid^n$ but nontrivial on $\aid^{n+1}$, the induced character of $\widetilde{\chi}$ in 
$\aid^{n+1}/\aid^n$ provides a number $\xi_n \in \{1,\ldots, e^{\Lambda(n+1)} - 1\}$ in a way that the finite ad\`ele $\xi_n e^{\psi(n)}$ represents a nonzero element such that  
\[ \chi( \xi_n e^{\psi(n)} x)=\chi(  e^{\psi(n)} x)^{\xi_n} =\widetilde{\chi}(x),\]
on $\aid^{n+1}$, where $\chi$ is the canonical character. Inductively, for any $m\geq n$, it is found a finite sum 
$$
S_m(\xi) = \sum_{k=n}^{m} \xi_k e^{\psi(k)} \qquad 
\left(\xi_k \in \{0,1,\ldots, e^{\Lambda(k+1)}-1\}, \, \xi_n \neq 0 \right)
$$ 
such that  $\widetilde{\chi}(x)= \chi\big( S_m(\xi) x \big)$ on $\aid^{m}$. Hence,
\[ \xi = \sum_{k=n}^{\infty} \xi_ke^{\psi(k)} \] 
is a finite ad\`ele such that $\tilde{\chi}(x) = \chi(\xi x)$. Finally, the pairing 
$\Af\times\Char{\Af}\To \C$ given by $(x,y)\mTo \chi_{y}(x)=\chi(xy)$ shows the topological group isomorphism $\Char(\Af)\cong\Af$.
\end{proof}

\begin{definition}
For any $q\in \Q$ the \emph{\textsf{annihilator}} of the subgroup $q\Zz \subset \Af$, in 
$\Char(\Af)$, is defined as the subgroup 
\[ \Ann(q\Zz ) = \left\{ \chi \in \Char(\Af) : \chi(x) = 1 \text{ for all } x\in q\Zz \right \}. \]
\end{definition}

\begin{proposition}
For any $q\in \Q$, the annihilator of $q\Zz$ is the subgroup $q^{-1}\Zz$. 
\end{proposition}

\begin{proof} 
For any nonzero rational number $q$, if $\xi \in q\Zz$ then $\xi q^{-1}  \Zz \subset \Zz$ and therefore $q^{-1}\Zz \subset \Ann(q \Zz)$. Now, if $\xi \notin q\Zz$, then $q^{-1} \xi \notin \Zz$. It follows that $\chi(q\xi) \neq 1$.
\end{proof}

\begin{corollary}
The annihilator of $\aid^n$ is the subgroup $\aid^{-n}$. In particular, $\Ann(\Zz)=\Zz$.
\end{corollary}

\begin{corollary}
There exists an isomorphism of topological groups
\[ \Char(\aid^n) \cong \Af/\aid^{-n} \qquad (n\in \Z). \]
\end{corollary}

\subsection[Integration theory on $\Af$ and some oscillatory integrals]{Integration theory on $\Af$ and some oscillatory integrals}
\label{integration_Af}

In this paragraph some aspects of the integration theory on the measure space $(\Af,dx)$ is presented. From the ultrametric $\dA$, a concept of improper integral arises and this gives sense to some oscillatory integrals on $\Af$.

\subsubsection[Change of variables formula]{Change of variables formula}  
\label{change_variable}

From Proposition \ref{ChangeVariables}, if $q$ is a nonzero rational number, the following \emph{change of variables} holds: 
\[ dx(qx) = q^{-1} x \] 
and
\[ \int_{K}f(x)dx = q^{-1} \int_{\frac{K-b}{q}} f(qx+b)dx, \]
where $K\subset \Af$ is a compact subset. 

As first examples the area of balls and spheres multiplied by a rational number are calculated in the sequel. For each $n\in \Z$, the ball $B_n$ centered at zero and radius $e^{\psi(n)}$ is precisely the subgroup $\aid^n$.

\begin{example}
If $qB_n = q\aid^n = qe^{-\psi(n)}\Zz$, then
$$
\int_{qB_n} dx = q^{-1}e^{\psi(n)} \int_{\Zz} d\mu(x) = q^{-1} e^{\psi(n)} \qquad (n\in \Z).
$$
\end{example}

\begin{example}
\begin{align*}
\int_{S_n} dx &=\int_{B_n\backslash B_{n-1}} dx\\
&= e^{\psi(n)} - e^{\psi(n-1 )}\\
&= e^{\psi(n)}(1 - e^{-\Lambda(n)}).
\end{align*}
\end{example}

\begin{example} 
From the above examples, for a radial function one obtains
\begin{align*}
\int_{\Af} f(\normAf{x}) dx
&= \sum_{n=-\infty}^{\infty} f(e^{\psi(n)}) e^{\psi(n)} (1 - e^{-\Lambda(n)}). 
\end{align*}
\end{example}

\begin{lemma} \emph{\textsf{(Integral Criterion)}}
If $f:\R_{\geq 0} \To \R_{\geq 0} \cup \{\infty\}$ is a non increasing continuous function, then
\[ \int_{\Af} f(\normAf{x}) d\mu(x) < \int_{0}^{\infty} f(t)dt. \]
\end{lemma}

\begin{proof} 
This follows from the classical integral criterion by observing that the area of balls equals to its radius:
$$
\int_{\Af} f(\normAf{x}) dx = 
\sum_{n=-\infty}^{\infty} f(e^{\psi(n)})(e^{\psi(n)} - e^{\psi(n-1)}) < \int_{0}^{\infty}f(t)dt. 
$$ 
\end{proof}

\subsubsection[Improper vs proper integrals]{Improper vs proper integrals}
\label{improper_proper-integrals}

If $f$ is an integrable function on $\Af$, then
\[ \int_{\Af}f(x)dx = \sum_{n=-\infty}^{\infty} \int_{S_n} f(x)dx. \] 
However, when considering some oscillatory integrals related to the ultrametric $\dA$ of $\Af$ which are not necessarily integrable functions, the definition involves the conditional limit
$$
\int_{\Af} f(x)dx = 
\lim_{n\to -\infty} \lim_{m\to \infty} \sum_{\gamma=n}^{m} \int_{S_\gamma}f(x) dx.
$$  

\begin{example}
\label{zeta_adelica}
For $\sigma=\Real(s)>0$, 
\begin{align*}
\int_{\Zz} \normAf{x}^{s-1} dx 
&= \sum_{n=0}^{\infty} e^{-(s-1)\psi(n)} e^{-\psi(n)} (1-e^{-\Lambda(n+1)}) \\
&= \sum_{n=0}^{\infty} \frac{1-e^{-\Lambda(n+1)}}{e^{s\psi(n)}} < \int_{0}^{1}x^{\sigma-1} dx.
\end{align*}
\end{example}

\begin{example}
\begin{align*}
-\int_{\Zz} \log(\normAf{x}) dx & = 
-\sum_{n=0}^{\infty} \log(e^{-\psi(n)}) e^{-\psi(n)}(1-e^{-\Lambda(n+1)}) \\
&=\sum_{n=0}^{\infty} \psi(n) e^{-\psi(n)}(1-e^{-\Lambda(n+1)}) \\
&< 1.
\end{align*} 
\end{example}

\subsubsection[Oscillatory integrals]{Oscillatory integrals}

Let $\chi$ be the canonical character of $\Af$ described in Section \ref{canonical_characterAf}. 

\begin{example} 
For $n\in \Z$,
$$
\int_{B_n} \chi(-\xi x) dx = 
\begin{cases}
e^{\psi(n)} & \text{ if } \normAf{\xi} \leq e^{-\psi(n)}, \\
0 & \text{ if } \normAf{\xi} > e^{-\psi(n)}.
\end{cases}
$$
In fact, if $\normAf{x} \leq e^{-\psi(n)}$, then $\chi(-\xi x)$ is the trivial character on 
$\aid^n$. On the other hand, if $\normAf{x} > e^{-\psi(n)}$, $\chi(-\xi x)$ is nontrivial on $\aid^n$ and there exists an element $b\in \aid^n$ such that 
$\chi(-\xi b)\neq 1$. Now, from the invariance of the Haar measure, one obtains
$$
\int_{\aid^n} \chi(-\xi x) dx = \int_{\aid^n} \chi(-\xi (x+b)) dx = 
\chi(-\xi b) \int_{\aid^n} \chi(-\xi x) dx.
$$
Hence $\int_{\aid^n} \chi(-\xi x) dx=0$ and the formula of the proposition holds.
\end{example}

This example implies:

\begin{example} 
For any $n\in \Z$, the following holds
$$
\int_{S_n} \chi(-\xi x) dx = 
\begin{cases}
e^{\psi(n)} - e^{\psi(n-1)}	& \text{ if } \normAf{\xi}  \leq e^{-\psi(n)}, \\
-e^{\psi(n-1)}						& \text{ if } \normAf{\xi} = e^{-\psi(n-1)}, \\
0 											& \text{ if } \normAf{\xi} \geq e^{-\psi(n-2)}.
\end{cases}
$$
\end{example}

\begin{example}
Let $f$ be a non increasing radial function on $\Af$, i.e. $f(x)=f(\normAf{x})$. For 
$\xi\in \Af$ write  $\gamma=\ord(\xi)$, then

\begin{align*}
\int_{\Af} f(x)\chi(-\xi x) dx 
&= \sum_{n=-\infty}^{\infty} \int_{S_n} f(\normAf{x}) \chi(-\xi x) dx \\
&= \int_{B_{\gamma}} f(x) d\mu(x)- f(e^{\psi(\gamma)})e^{\psi(\gamma+1)} \\
&= \sum_{n\leq\gamma}f(e^{\psi(n)})(e^{\psi(n)}-e^{\psi(n-1)}) - f(e^{\psi(\gamma)})e^{\psi(\gamma+1)} \\
&= \sum_{n \leq\gamma} e^{\psi(n)} \big( f(e^{\psi(n)}) - f(e^{\psi(n+1)}) \big). 
\end{align*}
\end{example}

\begin{example} 
The following conditional integral exists 
\[ \int_{\Af} \chi(-\xi x) dx = 0  \quad (\xi \neq 0). \]
\end{example}

\begin{example} 
\label{Riesz_kernel} 
This example is an adelic analogue of the $p$--adic Riesz kernel. For $\Real(s)>1$ and 
$\xi \in \Af$, with $\gamma=\ord(\xi)$, define $\Gamma(\xi,s)$ by the following conditional convergent integral
\begin{align*}
\int_{\Af} \normAf{x}^{s-1} \chi(-\xi x) dx 
&=\sum_{n \leq \gamma}e^{\psi(n)}(e^{(s-1)\psi(n)}-e^{(s-1)\psi(n+1)}) \\
&= \sum_{n \leq \gamma}e^{s\psi(n)}(1-e^{(s-1)\Lambda(n+1)}) \\
&= \sum_{n = -\gamma}^{\infty}e^{-s\psi(n)}\left(1-\frac{1}{e^{(s-1)\Lambda(n)}}\right).
\end{align*}
\end{example}

\begin{example}
\label{gamma_adelica}  
From this example the following adelic analogue of the $p$--adic $\Gamma$--function arises
\begin{align*}
\Gamma_{\Af}(s) &= \int_{\Af} \normAf{x}^{s-1} \chi(x) dx \\
&= \sum_{n=0}^{\infty}e^{-s\psi(n)} \big(1-\frac{1}{e^{(s-1)\Lambda(n)}} \big).
\end{align*}

\end{example}

\begin{example} 
If $\xi \neq 0$ and  $\gamma=\ord(\xi)$, we have

\begin{align*}
\int_{\Af} \log(\normAf{x})\chi(-\xi x) dx 
&= \sum_{n \leq \gamma}e^{\psi(n)}(\log(e^{\psi(n)}) - \log(e^{\psi(n+1)})) \\
&= \sum_{n \leq\gamma}e^{\psi(n)}(\psi(n) - \psi(n+1))\\
&= -\sum_{n \leq\gamma}e^{\psi(n)}\Lambda(n+1)\\
&= -\sum_{n=-\gamma}^{\infty}e^{-\psi(n)}\Lambda(n).
\end{align*} 
\end{example}

\begin{example} 
If $\xi \neq 0$ and $\gamma=\ord(\xi)$, then

\begin{align*}
\int_{\Af} \frac{\chi(-\xi x)}{\normAf{x}^2+M^2} dx 
&=\sum_{n\leq\gamma} e^{\psi(n)} \left(\frac{1}{e^{2\psi(n)}+M^2 } - 
\frac{1}{e^{2\psi(n+1)}+M^2 } \right)\\
&= 
\sum_{n=-\gamma}^{\infty} e^{-\psi(n)} \left(\frac{1}{e^{-2\psi(n)}+M^2 } - 
\frac{1}{e^{-2\psi(n-1)}+M^2 } \right) \notag, 
\end{align*}
\end{example}

\begin{remark}
The analogous integrals of all these examples in the classical $p$--adic case can be explicitly evaluated using the formula for a geometric series. Despite the fact that these integrals are not so easy to evaluate, the integral criterion implies that they are convergent.
\end{remark}

\begin{remark} 
Further analytic properties of general Dirichlet series in Example \ref{zeta_adelica}, Example \ref{Riesz_kernel} and Example \ref{gamma_adelica} beyond their given semiplane of convergence, are unknown to the authors (see  \cite{HR}).
\end{remark}

\section[Fourier analysis on $\Af$]{Fourier analysis on $\Af$}
\label{fourier_analysisAf}

The aim of this Section is to develop some basic ingredients of the Fourier transform on the space of test functions on $\Af$. The description of this theory in the $p$--adic case appears well documented in \cite{VVZ} and \cite{AKS}. The Bruhat--Schwartz space of a  locally compact commutative group has been introduced in \cite{Bru}. The general theory of topological vector spaces can be found in \cite{SW}.

\subsection[Bruhat--Schwartz functions]{Bruhat--Schwartz functions}
\label{BS_functionsAf}

Recall that $\aid^\ell$ denotes the ball of radius $e^{\psi(\ell)}$ and center at zero, 
$B_\ell(0)=e^{-\psi(\ell)} \Zz$. A function $\varphi:\Af \To \C$ is \textsf{locally constant} if for any $x\in \Af$, there exists an open subset $V_x \subset \Af$ such that 
$\varphi(x+y)=\varphi(x)$ for all $y\in V_x$. If $\varphi:\Af \To \C$ is a locally constant function, for any $x\in \Af$ there exists an integer $\ell(x) \in \Z$ such that
\[ \varphi(x+y) = \varphi(x), \quad \text{ for all } y\in \aid^{\ell(x)}.  \]

\begin{definition}
The \emph{\textsf{Bruhat--Schwartz space of finite adelic test functions}} $\D(\Af)$ is the space of locally constant functions on $\Af$ with compact support.
\end{definition}

Being $\Af$ a second countable totally disconnected locally compact topological group, the weakest topology on the vector space $\D(\Af)$ is a natural locally convex topology for which it is complete (see Remark \ref{topologyofAf} below). Let us describe this topology in terms of the ultrametric of $\Af$. If $\varphi \in \D(\Af)$ is different from zero, there exists a largest $\ell=\ell(\varphi)\in \Z$ such that, for every $x\in \Af$, 
\[ \varphi(x+y) = \varphi(x), \quad \text{ for all } \quad y \in \aid^{\ell}. \] 
This number $\ell$ is called \textsf{the parameter of constancy} of $\varphi$. Given two integers $\ell\leq k$, the collection of locally constant functions with support inside the compact ball $\aid^k$ and parameter of constancy $\ell$ will be denoted by $\Dkl$. 

Let $\car_\ell(x)$ be the characteristic (or indicator) function on the ball $\aid^\ell$:
$$
\car_\ell(x) =
\begin{cases}
1 \text{ if } x \in \aid^\ell, \\
0 \text{ if } x \notin \aid^\ell.
\end{cases}
$$

\begin{lemma}
The set $\Dkl$ forms a finite dimensional vector space over $\C$ of dimension 
$e^{\psi(k)}/e^{\psi(\ell)}$. Therefore it has a unique topology which makes it a topological vector space over $\C$.
\end{lemma}

\begin{proof} 
From the non--Archimedean property of $\Af$ it follows that any $\varphi\in \Dkl$ can be written as
$$
\varphi(x) = \sum_{\substack{a^u\in \aid^k/\aid^\ell}} \varphi(a^u) \car_{\ell}(x-a^u), 
$$
where  
${a^u}$ is a complete set of representatives of the classes of the quotient $\aid^k/\aid^\ell$. Since $\abs{\aid^k/\aid^\ell} = e^{\psi(k)}/e^{\psi(\ell)}$, it follows that 
$\Dkl$ has finite dimension equal to $e^{\psi(k)}/e^{\psi(\ell)}$.
\end{proof}

Let us notice that the following continuous inclusions are valid
$$ 
\Dkl \subset \D^{\ell'}_{k'}(\Af)\quad \text{ whenever } \quad k' \leq k, \, \ell \leq \ell'. 
$$ 

A partial order on the countable set of all pair of integers $(k,\ell)$, with $\ell\leq k$, is given by the folowinng relation: $(k',\ell')\leq (k,\ell)$, if and only if $k' \leq k, \, \ell \leq \ell'$. 
The topology of $\D(\Af)$ is described by the inductive limit
$$
\D(\Af) = \varinjlim_{\substack{\ell \leq k}} \Dkl,
$$ 
which  can also be written as the inductive limits
$$
\D^\ell(\Af) = \varinjlim_{k} \Dkl \qquad \text{ and } \qquad  
\D(\Af) = \varinjlim_{\ell} \D^\ell(\Af).
$$ 
Hence, a sequence of functions $(\varphi_j)_{j\geq 1}$ converges to $\varphi$ if there exist $\ell,k \in \Z $ such that $\varphi_j \in \Dkl$ for $j$ large enough and 
$\varphi_j \To \varphi$ on $\Dkl$. 

\begin{remark}
Since $\D(\Af)$ is the inductive limit of a countable family of (nuclear) finite dimensional vector spaces $\Dkl$, it is a complete locally convex topological vector space over $\C$ and a nuclear space.
\end{remark}
  
\begin{remark}
\label{topology_DAf}
The space $\D(\Af)$ and its topology only depends on the totally disconnected topology of $\Af$. In fact, as stated in \cite{Bru}, the Bruhat--Schwartz space of the locally compact Abelian topological group $\Af$ is given by
$$
\D(\Af)= \varinjlim_{K\subset H} \D(H/K),
$$ 
where $K\subset H \subset \Af$ are compact and open subgroups of $\Af$ and $\D(H/K)$ is the space of functions on $\Af$ with support on $H$ and constant on the cosets of $K$. Therefore any cofinal filtration with respect to the collection of all compact and open subgroups, $\{ q\Zz \}_{n\in \Q}$ of $\Af$ defines the same space $\D(\Af)$ with an  equivalent topology on it. Being a totally disconnected group, this topoloy depends only on the totally disconnected property \cite{Igu}.
\end{remark}

\begin{proposition}
For each compact subset $K\subset \Af$, let $\D(K)\subset \D(\Af)$ be the subspace of test functions with support on a fixed compact subset $K$ and let $\Co(K)$ be the space of complex valued continuous functions on $K$. The space $\D(K)$ is dense in $\Co(K)$, hence  $\D(\Af)$ is dense in $L^\rho(\Af)$  for $1\leq \rho < \infty$. Therefore 
$L^\rho(\Af)$ is separable for $1\leq \rho < \infty$.
 \end{proposition}

\begin{proof}
Let $\varphi\in \Co(K)$. For any $\epsilon >0$, there exists $\ell\in \Z$ such that $\left| \varphi(x)-\varphi(a)\right| \leq \epsilon$ if $x\in B_\ell(a)\cap K$ with $a\in K$. Since $K$ is compact, $K$ is the union of a disjoint set of balls $\{ B^\ell_j \}_{j=1}^{N}$. 
The characteristic functions of these balls $\car_{B^\ell_j}(x)$ are elements in $\D(\Qp)$ and satisfy
\[ \sum_{j=1}^{N} \car_{B^\ell_j}(x) = 1 \qquad (x\in K). \]

The function
\[ \varphi_\ell(x) = \sum_{j=1}^{N} \varphi(x_j) \car_{B^\ell_j}(x) \]
is an element of $\D(K)$. Therefore,
\begin{align*}
\norm{\varphi(x)-\varphi_\ell(x)}_{\Co(K)}	
&= \sup_{x\in K} \abs{\varphi(x)-\varphi_\ell(x)} \\
&= \sup_{x\in K} \abs{\varphi(x)-\sum_{j=1}^{N} \varphi(x_j) \car_{B^\ell_j}(x)} \\
&= \sup_{x\in K} \abs{\sum_{j=1}^{N} (\varphi(x) - \varphi(x_j))  \car_{B^\ell_j}(x)} \\
&\leq \sup_{x\in K} \sum_{j=1}^{N} \abs{\varphi(x) - \varphi(x_j)}  \car_{B^\ell_j}(x) \\
&\leq \epsilon  \sum_{j=1}^{N} \car_{B^\ell_j}(x) \\
&\leq \epsilon.
\end{align*} 

Finally, since the topology of $\Af$ is generated by a numerable set of balls, the spaces $L^\rho(\Af)$ and they are separable, for $1 \leq \rho < \infty$.   
\end{proof}

\subsection[The Fourier transform]{The Fourier transform}
\label{fourier_transformSB}

The Fourier transform of $\varphi \in \D(\Af)$ is given by
$$
\widehat{\varphi}(\xi) = \F[\varphi](\xi) = 
\int_{\Af} \varphi(x) \chi(\xi  x) dx, \qquad (\xi \in \Af).
$$ 

In the examples of oscillatory integrals, we have already computed the Fourier transform of some functions. In particular, the following property holds:

\begin{lemma} 
The Fourier transform of the characteristic function of the unit ball $\Zz$ coincides with itself. Moreover for any integer $\ell$,
\[ \F[\car_{\ell}(x)](\xi) =e^{\psi(\ell)} \car_{-\ell}(\xi), \]
\end{lemma} 

\begin{proposition}
\label{fourier_transformDkl} 
The Fourier transform is a linear application from $\Dkl$ to $\D^{-k}_{-\ell}(\Af)$. 
\end{proposition}

\begin{proof} 
For $\varphi \in \Dkl$ there is a representation
$$ 
\varphi(x) = \sum_{\substack{a^u \in \aid^k/\aid^\ell }} \varphi(a^u) \car_{\ell}(x-a^u), 
$$ 
where $a^u$ is a complete set of representatives of $\aid^k/\aid^\ell$. It is enough to take only characteristic functions $\car_{\ell}(x-a^u)$. Now
\begin{align*}
\F( \car_{\ell}(x-a^u))(\xi)	&= \int_{\Af} \chi(\xi  x) \car_{\ell}(x-a^u) d\mu(x) \\
&= \int_{\Af} \chi(\xi  (x+a^u)) \car_{\ell}(x) d\mu(x) \\
&= \chi(\xi a^u) \int_{\Af} \chi(\xi  x) \car_{\ell}(x) d\mu(x) \\
&= \chi(\xi a^u)e^{\psi(\ell)} \car_{-\ell}(\xi). 
\end{align*}

Hence, $\F(\car_{\ell}(x-a^u))$ has support in $\aid^{-\ell}$. Moreover, since 
$a^u\in \aid^k$, the character $\chi(\xi  a^u)$ is locally constant on balls of radius 
$ e^{\psi(-k)}$, i.e. it has rank at least $-k$. Therefore
\[ \F: \Dkl\To \D^{-k}_{-\ell}(\Af). \]
\end{proof}

\begin{theorem}
\label{inversion_formulaD}  
The Fourier transform is a continuous linear isomorphism from the space $\D(\Af)$ onto itself and the inversion formula holds:
\[ \varphi(x) = \int_{\Af} \chi(\xi x) \widehat{\varphi}(-\xi)d\xi, \qquad (\varphi\in \D(\Af)). \]
Additionally, the  Parseval--Steklov equality reads as
$$
\int_{\Af} \varphi(x) \overline{\psi(x)}dx = \int_{\Af} \widehat{\varphi}(\xi) \overline{\widehat{\psi}(\xi)} d\xi.
$$
\end{theorem}

\begin{proof} 
It has been already shown that $\F$ is a linear transformation from $\D(\Af)$ into itself. We show that the inversion formula holds. From Proposition \ref{fourier_transformDkl}, if 
$\varphi \in \Dkl$ then $\widehat{\varphi} \in \D^{-k}_{-\ell}(\Af)$. Hence

\begin{align*}
\int_{\Af} \chi(\xi x) \widehat{\varphi}(-\xi) d\xi 
&= \int_{\aid^{-\ell}} \chi(\xi x) \int_{\aid^k} \varphi(-y) \chi(\xi y) dyd\xi  \\
&= \int_{\aid^{k}} \varphi(-y) \int_{\aid^{-\ell}} \chi(\xi(x+y)) d\xi dy \\
&= \int_{\aid^{k}} \varphi(-y) e^{-\psi(\ell)} \car_\ell(x+y) dy \\
&= \sum_{\substack{a^u \in \aid^k/\aid^\ell}} \varphi(a^u) 
\int_{a^u + \aid^\ell} e^{-\psi(\ell)} \car_\ell(x+y) dy \\
&= \sum_{\substack{a^u \in \aid^k/\aid^\ell}} \varphi(a^u) \car_\ell(x-a^u) \\
&= \varphi(x).
\end{align*}

This shows that the Fourier transform $\F: \Dkl\To \D^{-k}_{-\ell}(\Af)$ is a continuous linear isomorphism, and therefore $\F$ is a continuous linear isomorphism from 
$\D(\Af)$ onto itself. From Fourier inversion formula and Fubini's theorem, the Parseval--Steklov identity follows.
\end{proof}

Recall that the Banach spaces $L^\rho(\Af)$ are separable for $ 1\leq \rho < \infty $. The Fourier transform of any integrable function can be defined by means of the same formula and the classical treatment of the Fourier theory on this space can be done as usual. In particular, $L^2(\Af)$ is a separable Hilbert space and the next statement holds: for 
$f\in L^2(\Af)$, define
\[ \F(f)(\xi) = \lim_{\ell \to \infty} \int_{\aid^\ell} \chi(x\xi) f(x) d\mu(x), \] 
where the limit is in $L^2(\Af)$.

\begin{theorem} 
The Fourier transform is a unitary transformation on the Hilbert space $L^2(\Af)$. The Fourier inversion formula and the Parseval--Steklov identity holds.
\end{theorem}

\section[Fourier Analysis on $\A$]{Fourier Analysis on $\A$}
\label{fourier_analysisA}

This Section extends the so far discussed results on Fourier analysis to the complete ring of ad\'eles $\A$. First, we recollect several results on the harmonic analysis of the Archimedean completion $\R$ of $\Q$.

\subsection[The Archimedean place]{The Archimedean place} 
\label{analysis_R}

Recall that the real numbers $\R$ is the unique Archimedean completion of the rational numbers. As a locally compact Abelian group, $\R$ is autodual with pairing function given by $\chi_\infty(\xi_\infty x_\infty)$, where $\chi_{\infty}(x_{\infty})=e^{-2\pi i x_{\infty}}$ is the canonical character on $\R$. In addition, it is a commutative Lie group. The Schwartz space of $\R$, which we denote here by $\D(\R)$, consists of functions 
$\varphi_\infty : \R \To \C$ which are infinitely differentiable and rapidly decreasing. 
$\D(\R)$ has a countable family of seminorms which makes it a nuclear Fr\'echet space. Let $dx_{\infty}$ denotes the usal Haar measure on $\R$. The Fourier transform 
$$
\F_{\R}[\varphi_\infty](\xi_\infty) = 
\int_\R \varphi_\infty(x_\infty) \chi_\infty(\xi_\infty x_\infty) dx_\infty 
$$
is an isomorphism from $\D(\R)$ onto itself. Moreover, the Fourier inversion formula and the Parseval--Steklov identities hold on $\D(\R)$. Furthermore, $L^2(\R)$ is a separable Hilbert space, the Fourier transform is and isometry on $L^2(\R)$, and the Fourier  inversion formula and the Parseval--Steklov identity holds on $L^2(\R)$.

\begin{definition}
The \emph{\textsf{ad\`ele ring}} $\A$ of $\Q$ is defined as $\A = \R\times \Af$. 
\end{definition}

With the product topology, $\A$ is a locally compact Abelian topological ring which admits a discrete inclusion of $\Q$. Since $\Af$ is totally disconnected, $(\R,0)\subset \A$ is the connected component of zero in $\A$. A Haar measure on $\A$ can be defined as the product measure $dx= dx_\infty  dx_f$, where $dx_f$ denotes the Haar measure on $\Af$.
 
\subsection[The group of Characters of $\A$]{The group of Characters of $\A$}
\label{charactersA}

If $\tilde{\chi}_{\infty}$ and $\tilde{\chi}_{f}$ are any characters on $\R$ and $\Af$, respectively, define a character $\tilde{\chi}$ on $\A$ by 
\[ \tilde{\chi}(x)=\tilde{\chi}(x_\infty)\tilde{\chi}(x_f) \qquad (x=(x_\infty,x_f)\in \A). \]
In particular, if $\chi_\infty$ and $\chi_f$ are the canonical characters on $\R$ and $\Af$, respectively, the expression
\begin{align*}
\chi(x)	&= \chi_\infty(x_\infty)\chi_f(x_f) \\
			&= \exp(-2\pi i x_\infty)  \exp(2\pi i \{x_f\}_{\Af}) \\
			&= \exp(2\pi i(-x_\infty + \{x_f\}_{\Af}))
\end{align*}
with $x=(x_\infty,x_f)\in \A$, defines a canonical character on $\A$.

Let $\Char(\A)$ be the Pontryagin dual group of $\A$. From the last identification it follows that $\R\times \Af \subset \Char(\A)$. Moreover, since the additive structure 
$\A=\R \times \Af$ is performed componentwise, it follows that
\[ \Char(\R\times \Af) \cong \Char(\R)\times \Char(\Af) \cong \R\times \Af, \]
and therefore, $\A$ is a selfdual group in the sense of Pontryagin. Moreover, for any prescribed ad\`ele $\xi=(\xi_\infty,\xi_f)$, a character $\chi_\xi$ of $\A$ is given by:
\begin{align*}
\chi_\xi(x)	&= \chi(\xi x) \\
					&= \chi_\infty(\xi_\infty x_\infty) \chi_f(\xi_f x_f) \\
					&= \exp(-2\pi i \xi_\infty x_\infty)  \exp(2\pi i \{\xi_f x_f\}_{\Af}) \\
					&= \exp\left( 2\pi i(-\xi_\infty x_\infty + \{\xi_f x_f\}_{\Af}) \right).
\end{align*}

\subsection[Bruhat--Schwartz space on $\A$]{Bruhat--Schwartz space on $\A$}
\label{BS_spaceA}
 
Recall that the Bruhat--Schwartz space of the locally compact Abelian group $\A$ is given as follows (see \cite{Bru}). First, notice that any compact subring of $\A$ is of the form 
$\{0\}\times K$, where $K$ is a compact and open subgroup of $\Af$. Since $\A$ is an autodual group any open and compactly generated subgroup of $\A$ is of the form
$\R\times H$, where $H$ is a compact and open subgroup of $\Af$. If $K\subset H$,
$\{0\}\times K$ is contained in $\R\times H$ and the quotient group 
$\R\times H/\{0\}\times K$ is an elementary group. By definition, 
$$ 
\D(\A)=\varinjlim_{\substack{ H \subset K}   } \D\left(\R \times H/ \{0\}\times K\right). 
$$ 

This allows to describe the topology of $\D(\A)$ using the filtration $\{\aid^n\}_{n\in \Z}$. First, for any $\varphi_\infty \in \D(\R)$ and $\varphi_f \in \D(\Af)$, define a function 
$\varphi$ on $\A$ by 
\[ \varphi(x) = \varphi_\infty(x_\infty)  \varphi_f(x_f) \]
for any ad\`ele $x=(x_\infty,x_f)$. These kind of functions are continuous on $\A$ and the  linear vector space generated by these functions is linearly isomorphic to the algebraic tensor product $\D(\R)\otimes \D(\Af)$. In the following, we identify these spaces and write $\varphi=\varphi_\infty \otimes \varphi_f$.

\begin{theorem}
The space of \emph{\textsf{Bruhat--Schwartz functions}} on $\A$ is the algebraic and topological tensor product of the nuclear vector spaces $\D(\R)$ and $\D(\Af)$, i.e.
\[ \D(\A) = \D(\R) \otimes \D(\Af). \]
\end{theorem}

\begin{proof} 
Since both $\D(\R)$ and $\D(\Af)$ are nuclear, the topological tensor product of these spaces is well defined. On the other hand, the Bruhat--Schwartz space $\D(\A)$ of the ring of ad\`eles  can be described as 
$$
\D(\A)=\varinjlim_{\substack{\ell,k \in\Z \\ \ell \leq k}   } \D\left(\R \times \aid^\ell /(0, \aid^k)\right). 
$$

Since $\R\times \aid^k/\{0\}\times\aid^\ell \cong \R\times (\aid^k/\aid^\ell) $ is an elementary group, for any integers $\ell \leq k$, $\D(\R \times\aid^k /(0, \aid^\ell) )$ is the algebraic and topological tensor product $\D(\R) \otimes \D(\aid^k/\aid^\ell)$ which is equal to $\D(\R) \otimes \Dkl$, where $\Dkl$ denotes the set of functions with support on $B_k\subset \Af$ and parameter of constancy $\ell$. Hence, since all the spaces involved in the next computations are nuclear, 
\begin{align}
\D(\A) &=\varinjlim_{\substack{\ell,k \in\Z \\ \ell \leq k} } \D(\R) \otimes \Dkl \notag\\
&= \D(\R) \otimes  \varinjlim_{\substack{\ell,k \in\Z \\ \ell \leq k} } \Dkl \notag\\
&= \D(\R) \otimes \D(\Af). \notag
\end{align}
\end{proof}

Let us further describe the topology of $\D(\A)$. Notice that the topological tensor product of $\D(\R)\otimes\Dkl$ is a Fr\'echet space, since $\D(\R)$ is Frechet and $\D_k^l(\Af)$, being finite dimensional, has only one seminorm. Therefore $\D(\A)$ is a nuclear LF space.

\begin{remark}
A sequence of functions $(\varphi_j)_{j\geq 1}$ in $\D(\A)$ converges to $\varphi$ if there exist $\ell,k \in \Z $ such that $\varphi_j \in \D(\R)\otimes\Dkl$ for $j$ large enough and 
$\varphi_j \To \varphi$ on the Fr\'echet space $\D(\R)\otimes\Dkl$. This can be split as
$$
\D^\ell(\A) = \varinjlim_k \D(\R)\otimes\Dkl \text{ and }\D(\A) 
= \varinjlim_\ell  \D^\ell(\A)
$$
\end{remark}
 
\subsubsection[The Fourier transform on $\A$]{The Fourier transform on $\A$}
\label{fourier_transformA}

The Fourier transforms on $\D(\R)$ and $\D(\Af)$ are, respectively, defined as:
\[ \F_{\R}[\varphi_\infty](\xi_\infty) = 
\int_\R \varphi_\infty(x_\infty) \chi_\infty(\xi_\infty x_\infty) dx_\infty \]
and
\[ \F_{\Af}[\varphi_f](\xi_f) = \int_{\Af} \varphi_f(x_f) \chi_f(\xi_f x_f) dx_f. \]

So the Fourier transform on $\D(\A)$ is described as
\[ \F[\varphi](\xi) = \int_{\A} \varphi(x) \chi(\xi x) dx, \]
for any $\xi\in \A$. It is well defined on $\D(\A)$ for any function of the form 
$\varphi=\varphi_\infty \otimes \varphi_f$ and it is given by
\[ \F[\varphi](\xi) = \F_\R[\varphi_{\infty}](\xi_\infty)\otimes \F_{\Af}(\varphi_f)(\xi_f) \] 
for any $\xi=(\xi_\infty,\xi_f)$. Succinctly, it can be written as
$\F_{\A} = \F_\R \otimes \F_{\Af}$. From Theorem \ref{inversion_formulaD} and the analogous result on the Archimedean place, the following holds.

\begin{theorem}
The Fourier transform $\F:\D(\A)\To \D(\A)$ is a linear and continuous isomorphism. The inversion formula on $\D(\A)$ reads as
\[ \F^{-1}[\varphi](\xi) = \int_{\A} \widehat{\varphi}(-\xi) \chi(\xi x) d\xi, \qquad 
(\xi \in \A), \]
and the Parseval--Steklov equality can be stated as
$$
\int_{\A} \varphi(x) \overline{\psi(x)}dx = \int_{\A} \widehat{\varphi}(\xi) \overline{\widehat{\psi}(\xi)} d\xi.
$$
\end{theorem}

The space of square integrable functions $L^2(\A)$ on $\A$ is a separable Hilbert space since it is the Hilbert tensor product space 
\[ L^2(\A) \cong L^2(\R) \otimes L^2(\Af) \] 
and $L^2(\R)$ and $L^2(\Af)$ are separable Hilbert spaces (see \cite{RS}). From this decomposition, the following fundamental and classical result is obtained.

\begin{theorem}
\label{fourier_transformA}
The Fourier transform $\F:L^2(\A)\To L^2(\A)$ is an isometry and the Fourier inversion formula and the Parseval--Steklov identities hold on $L^2(\A)$.
\end{theorem}

\section{Acknowledgements}

The first and third authors would like to thank Wilson A. Zu\~niga Galindo
for very useful discussions. Work was partially supported by   FORDECYT 265667.

\end{document}